\documentclass{amsart}
\usepackage[leqno]{amsmath}
\usepackage{amsfonts}
\usepackage{amssymb}
\usepackage{amsthm}
\usepackage[all]{xy}
\usepackage{graphicx, graphics}
\usepackage{ifsym}
\usepackage{textcomp, subfig, setspace, multicol,hyperref}

\theoremstyle{plain}
\newtheorem{theorem}{Theorem}[section]
\newtheorem{corollary}[theorem]{Corollary}
\newtheorem{lemma}[theorem]{Lemma}

\newtheorem{proposition}[theorem]{Proposition}

\theoremstyle{definition}

\newtheorem{example}[theorem]{Example}

\newcommand{\bk}[2]{B_{[#1,#2]}}

\newcommand{\precdot}{\smash{%
\raisebox{0ex}{%
\setlength{\tabcolsep}{-1.1pt}%
\begin{tabular}{@{}cc@{}}%
\:$\prec$&$\cdot$\:%
\end{tabular}}}}

\begin{document}
\title{The Shard Intersection Order on Permutations}
\author{Erin Bancroft}
\address{Department of Mathematics, North Carolina State University, Raleigh, NC 27695}
\email{erin\_bancroft@ncsu.edu}
\thanks{This is part of my dissertation at North Carolina State University supervised by Dr. Nathan Reading.}
\begin{abstract}
The shard intersection order is a new lattice structure on a finite Coxeter group $W$ which encodes the geometry of the reflection arrangement and the lattice theory of the weak order. In the case where $W$ is the symmetric group, we characterize shard intersections as certain pre-orders which we call permutation pre-orders. We use this combinatorial characterization to determine properties of the shard intersection order. In particular, we give an EL-labeling.
\end{abstract}
\maketitle
\section{Introduction}
\emph{Shards} were introduced in~\cite{nr03} as a way to understand lattice congruences of the weak order on a finite Coxeter group. They are defined in terms of the geometry of the associated simplicial hyperplane arrangement. The collection $\Psi$ of arbitrary intersections of shards studied in~\cite{nR09}, forms a lattice under reverse containment. This lattice is called the \emph{shard intersection order}. Surprisingly, $\Psi$ was found to be in bijection with the elements of the finite Coxeter group $W$, and thus the shard intersection order defines a new lattice structure on $W$. This lattice is graded and contains the $W$-noncrossing partition lattice NC($W$) as a sublattice. Indeed, for any Coxeter element, the subposet induced by $c$-sortable elements~\cite{nR07} is a sublattice isomorphic to NC($W$). A formula for calculating the M\"{o}bius numbers of lower intervals was given in~\cite{nR09}, but overall the structure of the shard intersection order is not yet well-understood.\\
\indent In this paper we consider the most classical Coxeter group, the symmetric group, whose associated hyperplane arrangement is the braid arrangement. In Section~\ref{sectshard}, we give necessary background information on hyperplane arrangements and the general construction of shards. We then specifically describe shards and shard intersections in the symmetric group. Throughout the paper, no prior knowledge of Coxeter groups will be assumed. Following a suggestion from Aguiar~\cite{mA}, in Section~\ref{sectpermpre} we characterize shard intersections of type $A$ by realizing them combinatorially as certain pre-orders, which we call \emph{permutation pre-orders}. In Section~\ref{sectintord}, we realize the shard intersection order as an order on the permutation pre-orders and use this realization to determine properties of the order, including an EL-labeling. Finally, in Section~\ref{sectnon} we characterize \emph{noncrossing pre-orders}, which correspond to $c$-sortable permutations.

\section{Shards and Shard Intersections}\label{sectshard}
In this section we begin by defining a central hyperplane arrangement and the construction of shards within it. We then focus on the symmetric group, starting with necessary background and concluding with an explicit description of the correspondence between shard intersections and permutations. A \emph{linear hyperplane} in a real vector space $V$ is a codimension-1 linear subspace of $V$. A \emph{central hyperplane arrangement} $\mathcal{A}$ in $V$ is a finite collection of linear hyperplanes. The \emph{regions} of $\mathcal{A}$ are the closures of the connected components of $V\setminus(\bigcup\mathcal{A})$. Each region is a closed convex polyhedral cone whose dimension equals dim($V$). \\
\indent Fix a \emph{base region} $B$ in the set of regions. We define a partial order on the set of regions called the \emph{poset of regions}. In the poset of regions, $Q$ is below $R$ if and only if the set of hyperplanes separating $Q$ from $B$ is contained in the set of hyperplanes separating $R$ from $B$. The unique minimal element in the poset is $B$ and the unique maximal element is $-B$, the region antipodal to $B$. A region $R$ covers $Q$ if and only if $R$ and $Q$ share a facet-defining hyperplane which separates $R$ from $B$ but does not separate $Q$ from $B$. (More information on the poset of regions can be found in \cite{aB90, pE84}.)\\
\indent A region is \emph{simplicial} if the normal vectors to its facet-defining hyperplanes form a linearly independent set. A central hyperplane arrangement is \emph{simplicial} if each of its regions is simplicial. A \emph{rank-two subarrangement} $\mathcal{A'}$ of $\mathcal{A}$ is a hyperplane arrangement consisting of all of the hyperplanes of $\mathcal{A}$ which contain some subspace of codimension-$2$, provided $|\mathcal{A'}|\geq 2$. In this subarrangement there exists a unique region $B'$ containing $B$, and the two facet-defining hyperplanes of $B'$ are called the \emph{basic hyperplanes} of $A'$.\\
\indent \emph{Shards} are defined by imposing a cutting relation on the hyperplanes in $\mathcal{A}$. For each nonbasic hyperplane $H$ in a rank-two subarrangement $\mathcal{A'}$, cut $H$ into connected components by removing the subspace $\cap\mathcal{A'}$ from $H$. A hyperplane $H$ is cut in every rank-two subarrangement in which it is non-basic, so it may be cut many times. Cutting every hyperplane $H$ in this way, we obtain a set of connected components, the closures of which are the \emph{shards} of $\mathcal{A}$. The set of \emph{intersections of shards} of an arrangement $\mathcal{A}$ is denoted $\Psi(\mathcal{A},B)$ or simply $\Psi$ when $\mathcal{A}$ and the choice of $B$ are clear. The empty intersection of shards is the entire space $V$.\\ 
\indent A shard $\Sigma$ is a \emph{lower shard} of a region $R$ if it is contained in a facet-defining hyperplane of $R$ which separates $R$ from a region covered by $R$. One of the primary results regarding shard intersections is a bijection~\cite[Proposition 4.7]{nR09} between regions of a simplicial hyperplane arrangement and intersections of shards. The bijection sends a region $R$ to the intersection of the lower shards of $R$. The shard intersections form a lattice under reverse containment, which induces, via the bijection, a partial order on the regions. Called the \emph{shard intersection order}, this partial order is different from the poset of regions and will be discussed further in Section~\ref{sectintord}. \\ 
\indent Now that we have considered the construction of shards in a general setting, let us turn to the symmetric group $S_n$. We begin with some background on $S_n$ and its Coxeter arrangement. Throughout this paper, permutations $\pi\in S_n$ will be written in one-line notation as $\pi=\pi_1\pi_2\cdots\pi_n$ with $\pi_i=\pi(i)$. An \emph{inversion} of $\pi$ is a pair $(\pi_i,\pi_j)$ such that $i<j$ and $\pi_i>\pi_j$. A \emph{descent} of $\pi$ is a pair $\pi_i\pi_{i+1}$ such that $\pi_i>\pi_{i+1}$. A \emph{descending run} of $\pi$ is a maximal descending sequence $\pi_i\pi_{i+1}\cdots\pi_{s}$. In this context, maximal implies that either $i=1$ or $\pi_i>\pi_{i-1}$ and either $s=n$ or $\pi_s< \pi_{s+1}$. For example, the permutation $\pi=1642735$ has descending runs $1, 642, 73$ and $5$.\\
\indent The Coxeter arrangement of the symmetric group consists of the hyperplanes $H_{ij}=\{\vec{x}\in\mathbb{R}^n:x_i=x_j\}$ for $1\leq i<j\leq n$. The permutations in $S_n$ are in bijection with the regions of $\mathcal{A}(S_n)$ as follows: the permutation $\pi\in S_n$ corresponds to the region $R_{\pi}=\{\vec{x}\in\mathbb{R}^n: x_{\pi_{1}} \leq x_{\pi_2} \leq \cdots \leq x_{\pi_n}\}$. We will choose the base region to be the one corresponding to $\pi=123\cdots n$. \\
\indent There are $2^{j-i-1}$ shards in each hyperplane $H_{ij}$. Each shard is obtained by choosing $\epsilon_k\in\{\pm 1\}$ for each $k$ with $i<k<j$, and defining $\Sigma$ to be the cone $\{\vec{x}: x_i=x_j\ \textrm{and}\ \epsilon_k x_i \leq \epsilon_kx_k\ \textrm{for}\ i<k<j\}$. An intersection of shards can be represented similarly as a set of vector equalities and inequalities, by taking the union of the equalities and inequalities defining the shards being intersected.
 
\begin{figure}[ht]
\[
\raisebox{-27 pt}{\includegraphics{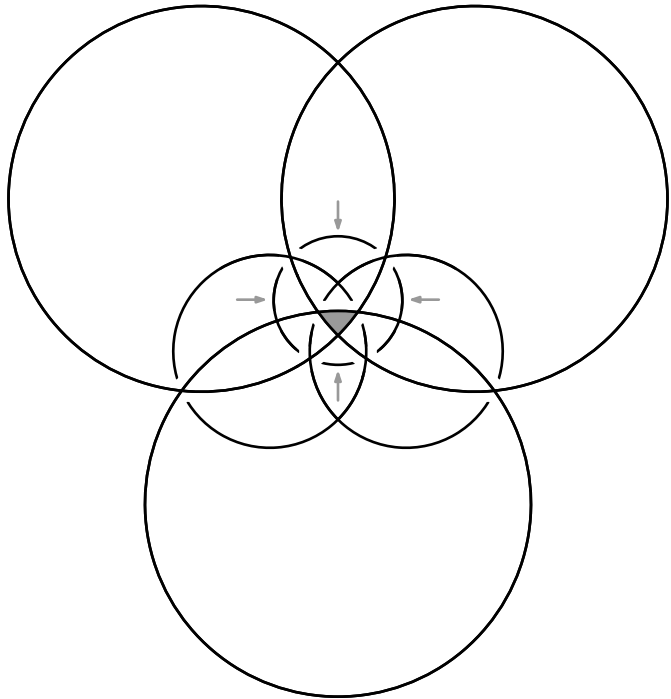}}
\]
\caption{Shards in the Coxeter arrangement $\mathcal{A}(S_4)$}\label{fig:s4shards}
\end{figure}

\begin{example}
The Coxeter arrangement $\mathcal{A}(S_4)$ consists of six hyperplanes through the origin in $\mathbb{R}^3$. These planes, intersected with the unit sphere in $\mathbb{R}^3$, define an arrangement of six great circles on the sphere. A stereographic projection yields an arrangement of six circles in the plane. This arrangement of circles is shown in Figure~\ref{fig:s4shards}. The three largest circles are the hyperplanes $H_{12}$ on the top left, $H_{23}$ on the bottom, and $H_{34}$ on the top right. The two medium circles are the hyperplanes $H_{13}$ on the left and $H_{24}$ on the right. The smallest circle is the hyperplane $H_{14}$. Regions of $\mathcal{A}(S_4)$ appear as curve-sided triangles. The base region, corresponding to $1234$, is shaded in gray. The shards are closed 2-dimensional cones, thus they appear as full circles or as circular arcs in the figure. To clarify the picture, where shards intersect, certain shards are offset slightly from the intersection to indicate that they do not continue through the intersection. The four shards contained in the hyperplane $H_{14}$ are marked by arrows in the figure. The top shard in $H_{14}$ is defined by $\{\vec{x}: x_1=x_4,\ x_1\leq x_2,\ -x_1\leq -x_3\}$. The left shard in $H_{14}$ is defined by $\{\vec{x}: x_1=x_4,\ x_1\leq x_2,\ x_1\leq x_3\}$. The bottom shard in $H_{14}$ is defined by $\{\vec{x}: x_1=x_4,\ -x_1\leq -x_2,\ x_1\leq x_3\}$. The right shard in $H_{14}$ is defined by $\{\vec{x}: x_1=x_4,\ -x_1\leq -x_2,\ -x_1\leq -x_3\}$.
\end{example}

Given a permutation $\pi$ and the corresponding region $R$, the descents of $\pi$ correspond to the hyperplanes containing the lower shards of $R$. For example, $43$ and $31$ are descents in $\pi=4312$ and the hyperplanes containing the shards below $R$ are $H_{13}$ and $H_{34}$. The cone formed by the intersection of the lower shards of $R$ satisfies $x_i=x_j$ for each descent $ji$ of $\pi$. Now we need to determine which shards from these hyperplanes are the lower shards of $R$. The shard below $R$ contained in $H_{ij}$ is the shard on the same side as $R$ of each hyperplane cutting $H_{ij}$. Thus, for each $k$ with $i<k<j$, the cone satisfies $x_i\geq x_k$ if and only if $(k,i)$ is an inversion of $\pi$, and the cone satisfies $x_i\leq x_k$ if and only if $(k,i)$ is not an inversion of $\pi$. Continuing the example, the shard below $R$ contained in $H_{13}$ will satisfy $x_1\leq x_2$, since $(2,1)$ is not an inversion in $\pi=4312$. The following proposition summarizes this explicit description for the shard intersection associated to a given permutation. 

\begin{proposition}\label{permtoshardint}
For a permutation $\pi$, the corresponding shard intersection is the cone consisting of points $(x_1,\ldots, x_n)$ satisfying the following conditions for each descent $ji$ of $\pi$.
\begin{enumerate}
\item $x_i\equiv x_j$ 
\item $x_i\geq x_k$ if and only if $i<k<j$ and $(k,i)$ is an inversion of $\pi$
\item $x_i\leq x_k$ if and only if $i<k<j$ and $(k,i)$ is not an inversion of $\pi$.
\end{enumerate}
\end{proposition}

We conclude this section with a lemma relating inversions of a permutation to intersections between descending runs. The notation $[a,b]$ stands for the set of integers $C$ with $a\leq c\leq b$ for $c\in C$. We will say that two descending runs $D=d_1d_2\cdots d_m$ and $E=e_1e_2\cdots e_n$ \emph{overlap} if $[d_m,d_1]\cap[e_n,e_1]\neq\emptyset$. This lemma will be used in Section~\ref{sectpermpre}. 

\begin{lemma}\label{inversionint}
Suppose $(i,j)$ is an inversion in $\pi$ such that $i$ and $j$ are in distinct descending runs $i_1i_2\cdots i_s$ and $j_1j_2\cdots j_t$ respectively. If $[i_s,i_1]\cap[j_t,j_1]=\emptyset$, then there exists a chain of descending runs $D_1, D_2,\ldots, D_n$ between $i_s$ and $j_1$ such that $[i_s,i_1]\cap D_1\neq\emptyset, D_n\cap[j_t,j_1]\neq\emptyset$, and $D_k$ and  $D_{k+1}$ overlap for $1\leq k \leq n-1$. 
\end{lemma}

\begin{proof}
Let $\pi=\cdots i_1i_2\cdots i_s\cdots j_1j_2\cdots j_t\cdots$ as in the statement of the lemma. Since $[i_s,i_1] \cap [j_t,j_1] = \emptyset$ and $i>j$, we know that $i_x>j_y$ for all $x\in [s]$ and $y\in [t]$. We claim there must be a descending run $d_1d_2\cdots d_p$ between $i_s$ and $j_1$ such that $d_1>i_s$ and $d_p<i_s$. We know $i_s>j_1$ and $i_s$ and $j_1$ are both part of different descending runs. Directly to the right of $i_s$ is the descending run $f_1\cdots f_r$ with $f_1>i_s$ since it is a descending run distinct from $i_1\cdots i_s$. Suppose $f_1\cdots f_r$ does not satisfy the claim. Then $f_z>i_s$ for all $z\in[r]$. Directly to the right of $f_r$ is the descending run $g_1\cdots g_p$ with $g_1>f_r$ since it is a distinct descending run. This implies $g_1>i_s$. Suppose $g_1\cdots g_l$ does not satisfy the claim. Then $g_z>f_r>i_s$ for all $z\in [l]$. We can continue considering adjacent descending runs until we reach the descending run $h_1\cdots h_w$ which is directly to the left of $j_1$. Again $h_1>i_s$ and if $h_1\cdots h_w$ does not satisfy the claim, then $h_z>i_s$ for all $z\in[w]$. From this we obtain $h_z>i_s>j_1$ for all $z\in[w]$ and in particular $h_w>j_1$. This implies that $h_1\cdots h_w$ and $j_1\cdots j_t$ are not separate descending runs, which is a contradiction. Therefore, at least one descending run $d_1\ldots d_p$ between $i_s$ and $j_1$ must satisfy the claim. Thus $[i_s,i_1]\cap[d_p,d_1]\neq\emptyset$. Let $d_1\cdots d_p$ be $D_1$. Also, since $d_1>i_s$, we know that $d_1>j_1$. Then, by induction on the number of descending runs between $i_s$ and $j_1$, there exists a chain of descending runs $D_1, D_2,\ldots, D_n$ between $i_s$ and $j_1$ such that $[i_s,i_1]\cap D_1\neq\emptyset, D_n\cap[j_t,j_1]\neq\emptyset$, and $D_k$ and $D_{k+1}$ overlap for $1\leq k \leq n-1$.
\end{proof}

\section{Permutation Pre-orders}\label{sectpermpre}

This section begins with background and notation concerning pre-orders, leading to an injection from shard intersections to pre-orders. Next, the bijection $\mu$ from permutations to pre-orders is defined and its image is characterized. We finish the section with the description of the inverse of $\mu$.\\
\indent A pre-order is a reflexive, transitive, binary relation. A pre-order $P$ on $[n]$ defines an equivalence relation on $[n]$ by setting $i\equiv j$ if and only if $i\preceq j \preceq i$. We will call the classes of this equivalence relation \emph{blocks}. We simultaneously think of $P$ as a pre-order on $[n]$ and as a partial order on its blocks. Notationally, when considering a specific pre-order $P$, a block will be labeled as $B_{[i,j]}$ where $i$ is the minimal element in the block and $j$ is the maximal element in the block. For example, a singleton block containing the element $l\in [n]$ will be labeled as $B_{[l,l]}$. $B(v)$ is defined to be the block in $P$ containing the element $v\in [n]$. We will say blocks $B_{[i,j]}$ and $B_{[k,l]}$ \emph{overlap} if the intervals $[i,j]$ and $[k,l]$ have a non-empty intersection. A block $B_{[i,j]}$ is covered by a block $B_{[k,l]}$ if $\bk{i}{j}\prec\bk{k}{l}$ and there does not exist a block $B_{[s,t]}$ such that $B_{[i,j]}\prec B_{[s,t]} \prec B_{[k,l]}$. To denote that block $\bk{k}{l}$ covers block $\bk{i}{j}$ we will write $\bk{i}{j}\precdot\bk{k}{l}$. An element $i$ is covered by an element $j$ in $P$ if $i\prec j$ and for any $s$ such that $i\preceq s \preceq j$, either $s\equiv i$ or $s \equiv j$.\\
\indent We now describe an injection from the set $\Psi$ of shard intersections in $\mathcal{A}(S_n)$ to the set $\mathcal{P}$ of pre-orders on $[n]$. The pre-order $P\in\mathcal{P}$ corresponding to a shard intersection $\Gamma\in\Psi$ is found by only considering the indices in the equalities and inequalities defining $\Gamma$. Thus for $i,j\in [n]$, $i\preceq j$ in $P$ if and only if the inequality $x_i\leq x_j$ holds in $\Gamma$. In particular, if $x_i=x_j$ in $\Gamma$ then $i\preceq j\preceq i$ in $P$.

\begin{figure}[ht]
\[
\raisebox{-27 pt}{\includegraphics{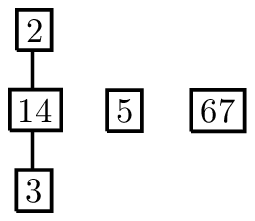}}
\]
\caption{$P$ for $\Gamma=\{\vec{x}: x_1=x_4,\ x_6=x_7,\ x_1\leq x_2,\ -x_1\leq -x_3\}\in S_7$}\label{shardint}
\end{figure}

\begin{example}
The shard intersection $\Gamma=\{\vec{x}: x_1=x_4,\ x_6=x_7,\ x_1\leq x_2,\linebreak -x_1\leq -x_3\}$ in $\mathbb{R}^7$ corresponds to the pre-order in Figure~\ref{shardint} which has blocks $B_{[3,3]}=\{3\},\ B_{[1,4]}=\{1,4\},\ B_{[2,2]}=\{2\},\ B_{[5,5]}=\{5\}$ and $B_{[6,7]}=\{6,7\}$. $\Gamma$ is the intersection of the shards $\{\vec{x}: x_1=x_4,\ x_1\leq x_2,\ -x_1\leq -x_3\}$ and $\{\vec{x}: x_6=x_7\}$ in $S_7$.
\end{example}

Let $\mu:S_n\to \mathcal{P}$ be the map that takes a permutation $\pi\in S_n$ to a pre-order $P \in \mathcal{P}$ as follows: each descending run $\pi_m\pi_{m+1}\cdots\pi_{m+t}$ in $\pi$ is a block $B_{[\pi_{m+t},\pi_m]}=\{\pi_m,\pi_{m+1},\ldots,\pi_{m+t}\}$ in $P$. For distinct blocks $B_{[\pi_{i+j},\pi_i]}$ and $B_{[\pi_{k+l},\pi_k]}$ which overlap, $B_{[\pi_{k+l},\pi_k]} \succeq B_{[\pi_{i+j},\pi_i]}$ in $P$ if and only if $\pi_k\cdots\pi_{k+l}$ occurs to the right of $\pi_i\cdots\pi_{i+j}$ in $\pi$, i.e.\ $(i+j)<k$. The transitive closure of these relations defines the pre-order $P$.

\begin{figure}[ht]
\[
\raisebox{-27 pt}{\includegraphics{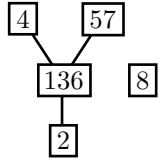}}
\]
\caption{$\mu(26314758)$}\label{exofmu}
\end{figure}

\begin{example}
For $\pi=26314758$ the descending runs form the following blocks: $B_{[2,2]}=\{2\}, B_{[1,6]}=\{1,3,6\}, B_{[4,4]}=\{4\}, B_{[5,7]}=\{5,7\}, B_{[8,8]}=\{8\}$. The pre-order $\mu(\pi)$ is shown in Figure \ref{exofmu}.
\end{example} 

\begin{proposition}\label{muiscomp}
$\mu$ is the composition of the bijection between permutations and shard intersections with the bijection between shard intersections and pre-orders.
\end{proposition}

\begin{proof}
All that needs to be shown is that the relations between elements established by $\mu$ correspond to the relations given in Proposition \ref{permtoshardint}. Suppose $ji$ is a descent in $\pi$. Then $i$ and $j$ are in the same descending run in $\pi$ and in $\mu(\pi)$ they will be in the same block, implying that $i\equiv j$. Suppose $k$ is an element of $\pi$ with $i<k<j$. In $\mu(\pi)$ the block $B(k)$ must be related to the block $B(i)$ since they overlap. If $(k,i)$ is an inversion in $\pi$, then we have $B(i)\succeq B(k)$ implying $i\succeq k$. If $(k,i)$ is not an inversion, then we have $B(i)\preceq B(k)$ implying $i\preceq k$. This corresponds to the relations given in Proposition~\ref{permtoshardint}.
\end{proof}

\indent Define a \emph{permutation pre-order} as a pre-order $P$ on $[n]$ with the following two conditions:
\begin{enumerate} 
\item[(P1)] if any two blocks in $P$ overlap, they must be comparable in $P$
\item[(P2)] all covering relationships in $P$ must be between blocks that overlap.
\end{enumerate} 
Let $\Omega=\{P: P\ \textrm{is a permutation pre-order}\}$.\\  
\indent As noted previously, a single shard $\Sigma\in\Psi$ can be represented as $\Sigma=\{\vec{x}: x_i=x_j\ \textrm{and}\ \epsilon_k x_i \leq \epsilon_kx_k\ \textrm{for}\ i<k<j\}$. Thus $\Sigma$ corresponds to a pre-order $P$ with one block $B_{[i,j]}$ of size two and $n-2$ singleton blocks: $B_{[v,v]}$ for $v\in[n]\setminus\{i,j\}$. Each block $B_{[k,k]}$ for $i<k<j$ will be comparable to $B_{[i,j]}$ in $P$ and all other blocks will be incomparable to it. All covering relationships in $P$ will involve overlapping blocks.  Therefore $P\in\Omega$. Define $\Omega^{\Sigma}\subset\Omega$ to be the set of permutation pre-orders corresponding to single shards.

\begin{proposition}\label{permtopreorder}
$\mu$ is a bijection from permutations in $S_n$ to pre-orders in $\Omega$. 
\end{proposition}

\begin{proof}
By Proposition \ref{muiscomp}, $\mu$ is a bijection onto its image. All that remains to be shown is that $\Omega$ is the image of $\mu$. It is clear from the definition of $\mu$ that the image of $\mu$ is contained in $\Omega$. Thus we will prove that $\Omega$ is contained in the image of $\mu$ by showing that every $\omega\in\Omega$ is an intersection of shards.\\
\indent Let $\omega\in\Omega$. It is sufficient to find a set of pre-orders in $\Omega^{\Sigma}$ such that the transitive closure of the unions of these pre-orders is $\omega$. For each $B_{[i,k]}$ in $\omega$ consider each $s\in B_{[i,k]}$ with $i<s$. Define $\omega_{is}$ to be the pre-order on $[n]$ with one two element block $B_{[i,s]}$ and the remaining elements in singleton blocks $B_{[v,v]}$ for $v\in [n]\setminus \{i,s\}$ with the following order relations:\\

(1) for $v\in [i,s]: \begin{cases} \text{if}\ i\preceq v\ \text{in}\ \omega, \text{then}\ i\preceq v\ \text{in}\ \omega_{is} \\
\text{if}\ i\succeq v\ \text{in}\ \omega, \text{then}\ i\succeq v\ \text{in}\ \omega_{is} \\
\text{if}\ i\equiv v\ \text{in}\ \omega, \text{then}\ i\preceq v\ \text{in}\ \omega_{is} \end{cases}$ \\
\indent ((P1) rules out the possibility that $i$ and $v$ are incomparable.)\\

\indent (2) for $v\notin [i,s]$: $v$ is incomparable in $\omega_{is}$ to $x$ for all $x\in [n]\setminus \{v\}$.\\ 

\indent Clearly, $\omega_{is}\in \Omega^{\Sigma}$. Taking the transitive closure of unions of $\omega_{is}$ for each pair $i,s$ we obtain all the blocks of $\omega$. (P2) guarantees that all of the covering relationships of $\omega$ will be obtained as well. Thus, taking the transitive closure of unions of the set of pre-orders $\omega_{is}$ (for each pair $i,s$ described above) gives $\omega$. 
\end{proof}

\begin{figure}[ht]
\[\omega\,=\raisebox{-27 pt}{\includegraphics{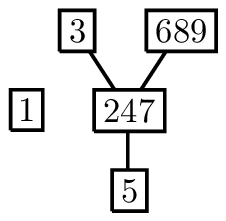}}\]\vspace{1em}
\begin{tabular}{lll}
$\omega_{24}\,=\raisebox{-27 pt}{\includegraphics{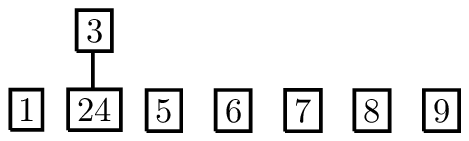}}$& \phantom{\hspace{.5cm}}&
$\omega_{27}\,=\raisebox{-27 pt}{\includegraphics{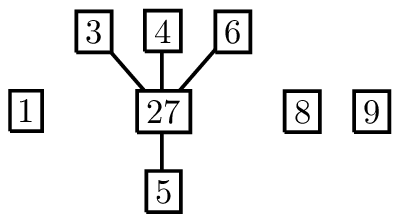}}$\\ 
$\omega_{68}\,=\raisebox{-27 pt}{\includegraphics{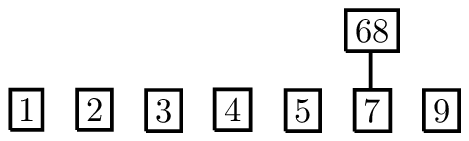}}$& \hspace{.4cm}&
$\omega_{69}\,=\raisebox{-27 pt}{\includegraphics{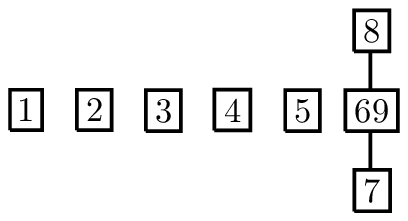}}$
\end{tabular}
\caption{}\label{ptp proof}
\end{figure}

\begin{example}
Figure \ref{ptp proof} illustrates the proof of Proposition \ref{permtopreorder}.
\end{example}

We now describe the inverse map to $\mu$. Define $\lambda: \Omega \to S_n$ as follows: given $\omega \in \Omega$, we define a permutation $\pi=\lambda(\omega)$ such that each block in $\omega$ is a descending run in $\pi$. For any blocks $B_1,B_2\in \omega$ if $B_1\prec B_2$ in $\omega$, then the descending run containing the elements of $B_2$ is to the right of the descending run containing the elements of $B_1$ in $\pi$. For any two incomparable blocks $B_{[i,j]},B_{[k,l]}\in \omega$ (by (P1) $\bk{i}{j}$ and $\bk{k}{l}$ do not overlap, thus without loss of generality $j<k$) the descending run containing the elements of $B_{[k,l]}$ is to the right of the descending run containing the elements of $B_{[i,j]}$ in $\pi$.

\begin{figure}[ht]
\[\raisebox{-27 pt}{\includegraphics{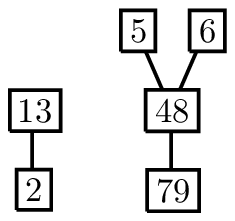}}
\]
\caption{}\label{exoflambda}
\end{figure}

\begin{example}
Suppose $\omega$ is the permutation pre-order shown in Figure \ref{exoflambda}. Then $\lambda(\omega)=231978456$. 
\end{example}

Next we give a lemma establishing that $\lambda$ is indeed a permutation and conclude the section with the proof that $\lambda=\mu^{-1}$.

\begin{lemma}\label{linear}
For any permutation pre-order $\omega$, the image $\lambda(\omega)$ is a permutation. 
\end{lemma}

\begin{proof}
The definition of $\lambda(\omega)$ defines a directed graph on the blocks of $\omega$ such that any two blocks are connected by a single directed edge. We will show that the digraph is acyclic, meaning that we have defined a total order on the blocks of $\omega$ and $\lambda(\omega)$ is a permutation. It is enough to show that any set of three blocks from $\omega$ is acyclically ordered by $\lambda$. Suppose $B_{[i,j]}, B_{[k,l]}, B_{[s,t]}\in \omega$. There are four cases to consider.\\
\emph{Case 1}: If $B_{[i,j]}, B_{[k,l]}$, and $B_{[s,t]}$ are all incomparable in $\omega$, then $[i,j], [k,l]$ and $[s,t]$ are pairwise disjoint, and $\lambda$ orders them numerically based on the positions of $[i,j], [k,l]$ and $[s,t]$ in $\mathbb{R}$. Thus, they will be ordered acyclically.\\ 
\emph{Case 2}: Suppose exactly two blocks are comparable. Without loss of generality let $\bk{i}{j} \prec \bk{k}{l}$ and let $\bk{s}{t}$ be incomparable to both $\bk{i}{j}$ and $\bk{k}{l}$. We claim that either $t<i$ and $t<k$ or $s>j$ and $s>l$. Since $B_{[s,t]}$ is incomparable to both $B_{[i,j]}$ and $B_{[k,l]}$ in $\omega$, we know that $[s,t]\cap[i,j]=\emptyset$ and $[s,t]\cap[k,l]=\emptyset$. There are two subcases to consider: either $[i,j]\cap[k,l]\neq\emptyset$ or $[i,j]\cap[k,l]=\emptyset$.\\
\indent \emph{Case 2a}: If $[i,j]\cap[k,l]\neq\emptyset$, then it follows that $[s,t]\cap[\min(i,k),\max(j,l)]=\emptyset$ and the claim is true.\\
\indent \emph{Case 2b}: If $[i,j]\cap[k,l]=\emptyset$, then by (P2) there exists a chain of covers $\bk{i}{j}\precdot B_1\linebreak \precdot B_2 \precdot \cdots \precdot B_n \precdot \bk{k}{l}$ in $\omega$ such that $\bk{i}{j}$ and $B_1$ overlap, $B_n$ and $\bk{k}{l}$ overlap, and $B_x$ and $B_{x+1}$ overlap for $1\leq x \leq n-1$. Since $\bk{s}{t}$ is incomparable to both $\bk{i}{j}$ and $\bk{k}{l}$, it must be incomparable to $B_x$ for $x\in[n]$. Therefore $B_x$ and $\bk{s}{t}$ do not overlap for $x\in[n]$. Again it follows that $[s,t]\cap[\min(i,k),\max(j,l)]=\emptyset$ and the claim is true.\\ 
\indent Since the claim is true, the definition of $\lambda$ implies that $\bk{s}{t}$ is to the left or to the right of both $\bk{i}{j}$ and $\bk{k}{l}$ in $\lambda(\omega)$. Thus the blocks will be ordered acyclically by $\lambda$.\\
\emph{Case 3}: Suppose two pairs of blocks are comparable and one pair is incomparable. Without loss of generality, either $\bk{s}{t}\prec \bk{i}{j}$ and $\bk{s}{t}\prec \bk{k}{l}$ or $\bk{i}{j}\prec\bk{s}{t}$ and $\bk{k}{l}\prec\bk{s}{t}$, with $\bk{i}{j}$ and $\bk{k}{l}$ incomparable in both instances. In the first instance, by the definition of $\lambda$, $\bk{s}{t}$ is to the left of both $\bk{i}{j}$ and $\bk{k}{l}$ in $\lambda$. Since $\bk{i}{j}$ and $\bk{k}{l}$ are incomparable in $\omega$, $[i,j]\cap[k,l]=\emptyset$ and the blocks $\bk{i}{j}$ and $\bk{k}{l}$ are ordered numerically in $\lambda$ based on the positions of $[i,j]$ and $[k,l]$ in $\mathbb{R}$. Thus the three blocks will be ordered acyclically by $\lambda$. The second instance follows similarly, and again the three blocks will be ordered acyclically by $\lambda$.\\
\emph{Case 4}: Finally, if $B_{[i,j]}, B_{[k,l]}$ and $B_{[s,t]}$ are all comparable in $\omega$, then by the transitivity of $\omega$ they must be ordered acyclically.\\
Therefore, $\lambda(\omega)$ defines a total order on the blocks of $\omega$, which corresponds to a unique permutation $\pi$.
\end{proof}

\begin{proposition}\label{inverse}
$\lambda=\mu^{-1}$
\end{proposition}

\begin{proof}
Since we know that $\mu$ is a bijection and we proved in Lemma \ref{linear} that $\lambda$ is well-defined, it is enough to show that $\lambda(\mu(\pi))=\pi$ to complete the proof. This is equivalent to showing that the relative positions of any two descending runs in $\pi$ remain the same in $\lambda(\mu(\pi))$. Descending runs in $\pi$ are preserved in $\lambda(\mu(\pi))$ because $\mu$ converts the descending runs to blocks and $\lambda$ converts each block back to the original descending run. (The descending runs do not increase in length since covering pairs $\bk{i}{j}$ and $\bk{k}{l}$ have $[i,j]\cap [k,l]\neq \emptyset$ and $\lambda$ places incomparable blocks in increasing order.) Let $i_1i_2\cdots i_s$ and $j_1j_2\cdots j_t$ be descending runs in $\pi$ with $\pi=\cdots i_1i_2\cdots i_s\cdots j_1j_2\cdots j_t\cdots$. Each of the descending runs in $\pi$ will become blocks in $\mu(\pi)$, so $B_{[i_s,i_1]}$ and $B_{[j_t,j_1]}$ are blocks in $\mu(\pi)$. There are two cases to consider.\\ 
\emph{Case 1}: Suppose $[i_s,i_1] \cap [j_t,j_1] \neq \emptyset$. Then $B_{[i_s,i_1]} \prec B_{[j_t,j_1]}$ in $\mu(\pi)$ and $\lambda$ will place the descending run $i_1i_2\cdots i_s$ to the left of the descending run $j_1j_2\cdots j_t$ in $\lambda(\mu(\pi))$. \\
\emph{Case 2}: Suppose $[i_s,i_1] \cap [j_t,j_1] = \emptyset$. Then we have two subcases to consider: either $i_x>j_y$ for all $x\in [s]$ and $y\in [t]$ or $i_x<j_y$ for all $x\in [s]$ and $y\in [t]$.\\ 
\indent \emph{Case 2a}: Suppose $i_x>j_y$ for all $x\in[s]$ and $y\in[t]$. Then by Lemma \ref{inversionint}, $B_{[i_s,i_1]}\prec B_{[j_t,j_1]}$ in $\mu(\pi)$ and $\lambda$ will place the descending run $i_1i_2\cdots i_s$ to the left of the descending run $j_1j_2\cdots j_t$ in $\lambda(\mu(\pi))$.\\
\indent \emph{Case 2b}: Suppose $i_x<j_y$ for all $x\in[s]$ and $y\in[t]$. For the purpose of contradiction, suppose that the descending run $i_1i_2\cdots i_s$ is to the right of the descending run $j_1j_2\cdots j_t$ in $\lambda(\mu(\pi))$. Since $i_x<j_y$ for $x\in[s]$ and $y\in[t]$, the descending runs were not ordered based on their positions in $\mathbb{R}$ implying that $\bk{j_t}{j_1}\prec\bk{i_s}{i_1}$ in $\mu(\pi)$. Since $[i_s,i_1] \cap [j_t,j_1] = \emptyset$, we know that $\bk{j_t}{j_1}$ is not covered by $\bk{i_s}{i_1}$. Thus we have a chain of covers $\bk{j_t}{j_1}\precdot B_1\precdot B_2\precdot \cdots \precdot B_n\precdot \bk{i_s}{i_1}$ in $\mu(\pi)$ where $n\geq 1$. Since $\bk{j_t}{j_1}\precdot B_1$ in $\mu(\pi)$, the descending run $j_1\cdots j_t$ must be to the left of the descending run containing the elements of $B_1$ in $\pi$. Similarly, the descending run containing the elements of $B_x$ must be to the left of the descending run containing the elements of $B_{x+1}$ in $\pi$ for $1\leq x \leq n-1$ and the descending run containing the elements of $B_n$ must be to the left of the descending run $i_1\cdots i_s$ in $\pi$. Thus $j_1j_2\cdots j_t$ is to the left of $i_1i_2\cdots i_s$ in $\pi$, which is a contradiction. Therefore the descending run $i_1i_2\cdots i_s$ is to the left of the descending run $j_1j_2\cdots j_t$ in $\lambda(\mu(\pi))$.
\end{proof}

\section{The Shard Intersection Order on $S_n$}\label{sectintord}

In this section we describe the shard intersection order on $S_n$ in terms of an order on the permutation pre-orders. We then define a labeling for the order and spend the majority of the section proving various lemmas concerning cover relations and the labeling. The section, and the paper, culminate in the proof that the given labeling is an EL-labeling of the shard intersection order on $S_n$.\\
\indent The \emph{shard intersection order} is the set $\Psi(\mathcal{A},B)$ partially ordered by reverse containment. It was shown in \cite{nR09} to be a graded, atomic and coatomic lattice and is denoted as $(\Psi(\mathcal{A},B),\supseteq)$. We will continue considering the case where $W=S_n$. Thus $(\Psi(\mathcal{A}(S_n),B),\supseteq)$ can be realized as the permutation pre-orders partially ordered by containment of relations, which we will denote as $(\Omega,\leq_S)$. To avoid confusion, we remind the reader that the three order relations appearing in the paper are the total order $\leq$ on real numbers, the partial order $\preceq$ on blocks in a given permutation pre-order and the order relation $\leq_S$ on permutation pre-orders in $(\Omega,\leq_S)$. The minimal element $\hat{0}$ is the empty intersection, which is the permutation pre-order with each number in a singleton block and no order relations among the blocks. The maximal element $\hat{1}$ is the permutation pre-order with all numbers together in one block. In \cite{nR09} it was shown that the rank of $\Gamma\in\Psi$ is the codimension of $\Gamma$. From this, one easily deduces the following result for permutation pre-orders. The notation $|\omega|$ stands for the number of blocks in $\omega$.

\begin{proposition}
The rank of $\omega\in\Omega$ in the graded lattice $(\Omega,\leq_S)$ is $n -|\omega|$. 
\end{proposition}  

To go up by a cover from a shard intersection $\Gamma$ in $(\Psi(\mathcal{A},B),\supseteq)$ we intersect $\Gamma$ with one additional shard $\Sigma$, chosen so that $\dim(\Sigma\cap\Gamma)=\dim(\Gamma)-1$. Thus to go up by a cover from $\omega$ in $(\Omega,\leq_S)$, any two blocks which are unrelated or related by a cover can be combined. All relations from $\omega$ are maintained in the permutation pre-order that covers it, and relations involving either of the combined blocks from $\omega$ become an order relation involving the new block in the cover. If by combining two unrelated blocks in $\omega$ an interval is formed which now intersects the intervals of other unrelated blocks, then the newly overlapping blocks may be greater than or less than the combined block in the cover. \\

\begin{figure}[ht]
\[\raisebox{-27 pt}{\includegraphics{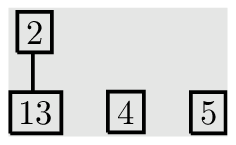}}\]
\[\omega\]\\
\vspace{.75cm}
\begin{tabular}{ccccccc}
$\raisebox{-27 pt}{\includegraphics{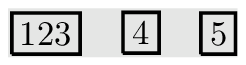}}$&\hspace{.2cm} &
$\raisebox{-27 pt}{\includegraphics{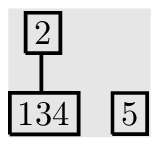}}$&\hspace{.5cm}&
$\raisebox{-27 pt}{\includegraphics{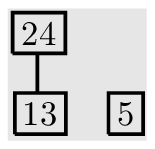}}$&\hspace{.5cm}&
$\raisebox{-27 pt}{\includegraphics{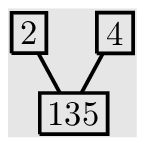}}$\\
\footnotesize{\textsc{(a)}}&\hspace{.2cm}&\footnotesize{\textsc{(b)}}&\hspace{.5cm}&\footnotesize{\textsc{(c)}}&\hspace{.5cm}&\footnotesize{\textsc{(d)}}
\vspace{.5cm}\\
$\raisebox{-27 pt}{\includegraphics{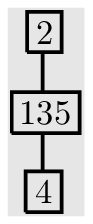}}$&\hspace{.2cm}&
$\raisebox{-27 pt}{\includegraphics{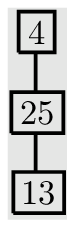}}$&\hspace{.5cm}&
$\raisebox{-27 pt}{\includegraphics{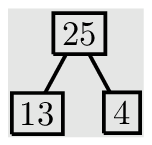}}$&\hspace{.5cm}&
$\raisebox{-27 pt}{\includegraphics{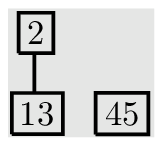}}$\\
\footnotesize{\textsc{(e)}}&\hspace{.2cm}&\footnotesize{\textsc{(f)}}&\hspace{.5cm}&\footnotesize{\textsc{(g)}}&\hspace{.5cm}&\footnotesize{\textsc{(h)}}
\end{tabular}
\caption{}\label{fig:covers}
\end{figure}

\begin{example}
Figure \ref{fig:covers} shows a permutation pre-order labeled $\omega$ and the eight permutation pre-orders labeled \textit{(A)} through \textit{(H)} which cover it. Notice that if we combine blocks $\bk{1}{3}$ and $\bk{5}{5}$ in $\omega$, blocks $\bk{1}{5}$ and $\bk{4}{4}$ overlap in the permutation pre-order which covers $\omega$. By (P2) this means that $\bk{1}{5}$ and $\bk{4}{4}$ must be related. Since $\bk{4}{4}$ is unrelated to blocks $\bk{1}{3}$ and $\bk{5}{5}$ in $\omega$ there is no restriction on the direction of this relation, so $\bk{4}{4}$ can be either greater than or less than $\bk{1}{5}$. Thus we obtain the two permutation pre-orders \textit{(D)} and \textit{(E)}. Each of these pre-orders arise by intersecting $\omega$ with a permutation pre-order corresponding to a single shard in $\Omega^{\Sigma}$ with $3$ and $5$ together in one block.\\
\indent In $\omega$ any two blocks can be combined to form a permutation pre-order which covers $\omega$. On the other hand, in the permutation pre-order labeled \textit{(F)}, blocks $\bk{1}{3}$ and $\bk{4}{4}$ cannot be combined to form a cover since this would force all numbers to be in one block and the rank would go up by two instead of one.
\end{example}

The main result of this paper is an edge-lexicographic or EL-labeling \cite{aB80,aB83} of $(\Omega,\leq_S)$. A sequence $a_1a_2\cdots a_n$ is \emph{lexicographically smaller} than a sequence $b_1b_2\cdots b_n$ if there exists a $j\in[n]$ such that $a_i=b_i$ for all $i<j$ but $a_j<b_j$. Define a labeling of the poset to be an \emph{EL-labeling} if the edges of the poset are labeled with positive integers such that the following two conditions hold.
\begin{enumerate}
\item[(EL1)] Each interval $(\omega,\omega')$ in the poset has a unique maximal chain $\omega=\omega_0\lessdot_S\omega_1\lessdot_S\cdots\lessdot_S\omega_k=\omega'$ whose edge label sequence, $\sigma(\omega_0,\omega_1)\sigma(\omega_1,\omega_2)\cdots\linebreak\sigma(\omega_{k-1},\omega_k)$ is weakly increasing in value, and
\item[(EL2)] this weakly increasing edge label sequence is lexicographically smaller than the label sequences of all other maximal chains from $\omega$ to $\omega'$.
\end{enumerate}

Given a permutation pre-order $\omega$, number the descending runs in $\lambda(\omega)$ from $1$ to $|\omega|$, proceeding from left to right. Define this number to be the \emph{placement} of the descending run $i_1\cdots i_n$ in $\lambda(\omega)$ or the \emph{placement} of its corresponding block $\bk{i_n}{i_1}$ in $\omega$. We will denote the placement of a block $B$ in $\omega$ as $pl_{\omega}(B)$. The total order on the blocks of $\omega$ defined by the placements is equivalent to the total order defined by $\lambda$. Also, if $\bk{i}{j}\prec\bk{k}{l}$ in $\omega$, then $pl_{\omega}(\bk{i}{j})<pl_{\omega}(\bk{k}{l})$ since the descending run corresponding to $\bk{i}{j}$ is to the left of the descending run corresponding to $\bk{k}{l}$ in $\lambda(\omega)$.\\
\indent For $(\Omega,\leq_S)$ we will describe a labeling $\sigma$. Let $\sigma(\omega_i,\omega_{i+1})=$ the maximum of the placements of the two blocks from $\omega_i$ that were combined to form $\omega_{i+1}$. We will prove the following theorem.

\begin{theorem}\label{ELlabel}
$\sigma$ is an EL-labeling. 
\end{theorem}

\begin{figure}[ht]
\[\raisebox{-27 pt}{\includegraphics{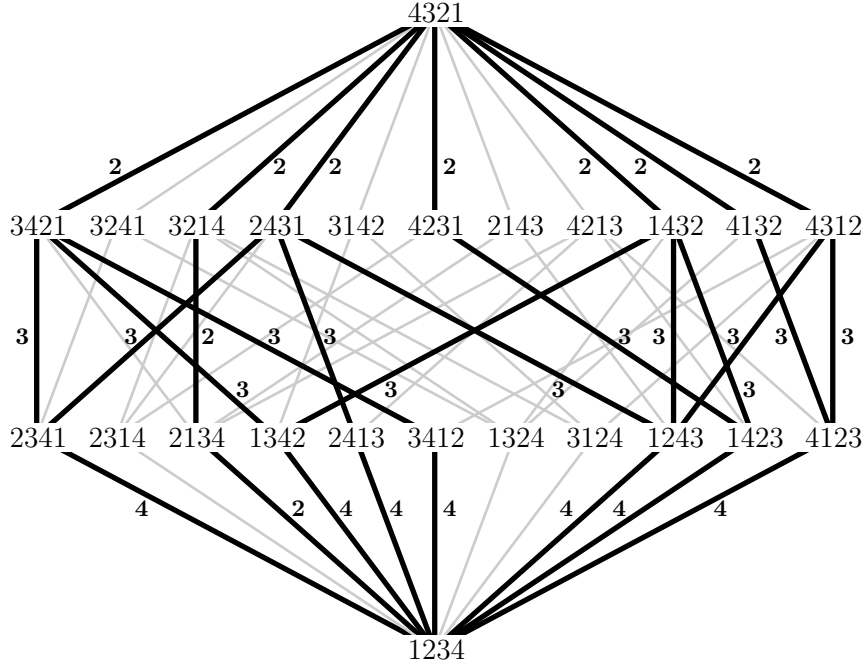}}
\]
\caption{The shard intersection order on $S_4$}\label{exofELlabel}
\end{figure}

\begin{example}  
The unique increasing maximal chain and the thirteen strictly decreasing maximal chains in the shard intersection order on $S_4$ with respect to the labeling $\sigma$ are shown in Figure \ref{exofELlabel}.
\end{example}

The proof of Theorem \ref{ELlabel} requires a proposition and several lemmas concerning relations in $(\Omega,\leq_S)$. Given a permutation pre-order $\omega$ and the interval $(\omega,\omega')$ in $(\Omega,\leq_S)$, we will denote by $T(\omega,\omega')$ the set of pairs of blocks in $\omega$ that could be combined to form a permutation pre-order in $(\omega,\omega')$ which covers $\omega$. A pair of blocks is in $T(\omega,\omega')$ if and only if the pair is in the same block in $\omega'$ and is either incomparable or related by a cover in $\omega$. A pair of blocks in $\omega$ will be called \emph{combinable} if the pair is in $T(\omega,\omega')$.

\begin{proposition}\label{superprop}
Let $\omega$ and $\omega'$ be permutation pre-orders with $\omega\lessdot_S\omega'$. Suppose $B_1$ and $B_2$ are the blocks in $\omega$ that are combined to form a block $B$ in $\omega'$, with $pl_{\omega}(B_1)=a<c=pl_{\omega}(B_2)$. Let $B^*$ be in $\omega$.
\begin{enumerate}
\item If $a<pl_{\omega}(B^*)<c$, then $B^*$ is comparable to $B$ in $\omega'$.

\item\label{btwn2} If $B^*$ is incomparable to $B_1$ and $B_2$ in $\omega$ and comparable to $B$ in $\omega'$, then $a<pl_{\omega}(B^*)<c$.

\item\label{labelorder} If $pl_{\omega}(B^*)<a$, then $pl_{\omega'}(B^*)=pl_{\omega}(B^*)$. If $a\leq pl_{\omega}(B^*)\leq c$, then $a\leq pl_{\omega'}(B^*)\leq c-1$. If $pl_{\omega}(B^*)>c$, then $pl_{\omega'}(B^*)=pl_{\omega}(B^*)-1$.

\item\label{comb} Let $B'$ be a block in $\omega'$ such that $pl_{\omega'}(B')=b<c$. If blocks $B$ and $B'$ are combinable in $\omega'$, then $B'$ is combinable with $B_1$ in $\omega$.

\item\label{staysmall} If $pl_{\omega}(B^*)=b<c$, then $B'$, the block containing $B^*$ in $\omega'$, has $pl_{\omega'}(B')<c$. 
\end{enumerate}
\end{proposition}

\begin{proof}
Let $B_1=\bk{i}{j}$ and $B_2=\bk{k}{l}$. Then $B=\bk{s}{t}$ for $s=\min(i,k)$ and $t=\max(j,l)$.\\
\emph{Proof of (1)}: There are two cases to consider. Suppose, in the first case, that $\bk{i}{j}$ and $\bk{k}{l}$ form a cover in $\omega$. Arguing as in Case 2a of Lemma \ref{linear}, any block which is incomparable to both $\bk{i}{j}$ and $\bk{k}{l}$ must be to the left of both $\bk{i}{j}$ and $\bk{k}{l}$ or to the right of both $\bk{i}{j}$ and $\bk{k}{l}$ in $\lambda(\omega)$ and thus does not have a placement between $a$ and $c$. Therefore, all blocks $B^*$ with $a<pl_{\omega}(B^*)<c$ must be comparable to either $\bk{i}{j}$ or $\bk{k}{l}$ in $\omega$. Since all relations from $\omega$ must also be in $\omega'$, the blocks $B^*$ must be comparable to $\bk{s}{t}$ in $\omega'$.\\
\indent In the second case we suppose that $\bk{i}{j}$ and $\bk{k}{l}$ are incomparable in $\omega$. Again, any block $B^*$ with $a<pl_{\omega}(B^*)<c$ which is comparable to either $\bk{i}{j}$ or $\bk{k}{l}$ in $\omega$ must be comparable to $\bk{s}{t}$ in $\omega'$. Thus, all that is left to consider is a block $\bk{p}{q}$ with $a<pl_{\omega}(\bk{p}{q})<c$ which is incomparable to both $\bk{i}{j}$ and $\bk{k}{l}$ in $\omega$. Since the three blocks $\bk{i}{j}, \bk{k}{l}$ and $\bk{p}{q}$ are incomparable in $\omega$, we know that $[i,j], [k,l]$ and $[p,q]$ are pairwise disjoint and the three blocks are ordered in $\lambda(\omega)$ based on the positions of $[i,j], [k,l]$ and $[p,q]$ in $\mathbb{R}$. Given that the placements of the three blocks correspond to their ordering in $\lambda(\omega)$ and $pl_{\omega}(\bk{i}{j})<pl_{\omega}(\bk{p}{q})< pl_{\omega}(\bk{k}{l})$, we can conclude that $i\leq j< p\leq q < k\leq l$. Combining $\bk{i}{j}$ and $\bk{k}{l}$ forms the block $\bk{i}{l}=\bk{s}{t}$, and we know that $[p,q]\subset[i,l]=[s,t]$. Therefore by (P1), $\bk{p}{q}$ must be related to $\bk{s}{t}$ in $\omega'$.\\
\emph{Proof of (2)}:
Suppose for the purpose of contradiction, a block $\bk{f}{g}$ with \linebreak $pl_{\omega}(\bk{f}{g})<a$ is incomparable to both $\bk{i}{j}$ and $\bk{k}{l}$ in $\omega$ and is comparable to $\bk{s}{t}$ in $\omega'$. Since $\lambda$ orders incomparable blocks based on their intervals, $g<i$ and $g<k$. Thus, $g<\min(i,k)$ and $[f,g]\cap[s,t]=\emptyset$. This implies that $\bk{f}{g}$ and $\bk{s}{t}$ do not cover one another in $\omega'$ and therefore must be related by a chain of covers: $\bk{f}{g}\precdot \bk{x_1}{x_2}\precdot \bk{x_3}{x_4} \precdot \cdots \precdot \bk{x_{d-1}}{x_d} \precdot\bk{s}{t}$. By (P2), we have $[f,g]\cap[x_1,x_2]\neq \emptyset, [x_1,x_2]\cap[x_3,x_4]\neq\emptyset, \ldots ,[x_{d-1},x_d]\cap[s,t]\neq\emptyset$. Thus, some interval $[x_m,x_{m+1}]$ for $1\leq m\leq d-1$ must overlap either $[i,j]$ or $[k,l]$. The block $\bk{x_m}{x_{m+1}}$ corresponding to this interval must be less than either $\bk{i}{j}$ or $\bk{k}{l}$ in $\omega$ (whichever it overlaps) because if it were greater than either of them it would be greater than $\bk{s}{t}$ in $\omega'$ and thus would not be part of the chain of covers. Since all blocks except for $\bk{i}{j}$ and $\bk{k}{l}$ are the same in both $\omega$ and $\omega'$, we have that $\bk{f}{g}$ is related to either $\bk{i}{j}$ or $\bk{k}{l}$ by a chain of covers in $\omega$. This contradicts our assumption that $\bk{f}{g}$ is incomparable to both $\bk{i}{j}$ and $\bk{k}{l}$ in $\omega$ and thus $pl_{\omega}(\bk{f}{g})\nless a$. Similarly, if $pl_{\omega}(\bk{f}{g})>c$ we see that $f>\max(j,l)$ and there must exist a chain of covers between $\bk{s}{t}$ and $\bk{f}{g}$ such that one block in the chain has an interval that overlaps either $[i,j]$ or $[k,l]$. Reaching the same contradiction we conclude that $pl_{\omega}(\bk{f}{g})\ngtr c$ and thus $a<pl_{\omega}(\bk{f}{g})<c$.\\
\emph{Proof of (3)}:
Let $\bk{u}{v}$ be a block in $\omega$ with $pl_{\omega}(\bk{u}{v})=x<a$. We first show that $pl_{\omega'}(\bk{u}{v})<pl_{\omega'}(\bk{s}{t})$. Suppose that $\bk{u}{v}$ is incomparable to both $\bk{i}{j}$ and $\bk{k}{l}$. Since $x<a$ and $x<c$, $v<\min(i,k)=s$ and the descending run corresponding to $\bk{s}{t}$ must be to the right of the descending run corresponding to $\bk{u}{v}$ in $\lambda(\omega')$. Now suppose that $\bk{u}{v}$ is comparable to $\bk{i}{j}$ or $\bk{k}{l}$ in $\omega$. Then either $\bk{u}{v}\prec\bk{i}{j}$ or $\bk{u}{v}\prec\bk{k}{l}$, since $x<a$ and $x<c$. All relations from $\omega$ must also hold in $\omega'$ so $\bk{u}{v}\prec\bk{s}{t}$ in $\omega'$. Thus the descending run corresponding to the new block $\bk{s}{t}$ will always be to the right of the descending run corresponding to $\bk{u}{v}$ in $\lambda(\omega')$.\\
\indent We next show that $pl_{\omega}(\bk{u}{v})=pl_{\omega'}(\bk{u}{v})$. Let $\bk{f}{g}$ be any block in $\omega$ other than $\bk{i}{j}, \bk{k}{l}$ and $\bk{u}{v}$. If $\bk{u}{v}$ is incomparable to $\bk{f}{g}$ in both $\omega$ and $\omega'$ then the descending runs corresponding to $\bk{u}{v}$ and $\bk{f}{g}$ must be ordered the same in both $\lambda(\omega)$ and $\lambda(\omega')$, since $\lambda$ orders them based on the positions of $[u,v]$ and $[f,g]$ in $\mathbb{R}$. If $\bk{u}{v}$ is comparable to $\bk{f}{g}$ in $\omega$, then again the descending runs corresponding to each of them must be ordered the same in both $\lambda(\omega)$ and $\lambda(\omega')$, since relations from $\omega$ are preserved in $\omega'$. If $\bk{u}{v}$ is incomparable to $\bk{f}{g}$ in $\omega$ and comparable to it in $\omega'$, then we claim that $v<f$ and  $\bk{u}{v}\prec\bk{f}{g}$ in $\omega'$. Two blocks which are comparable in $\omega'$ and incomparable in $\omega$ can only occur when both blocks are comparable through the new block $\bk{s}{t}$ in $\omega'$. Since we already showed that $\bk{u}{v}$ must be less than $\bk{s}{t}$ in $\omega'$ if they are comparable, we must have $\bk{u}{v}\prec\bk{s}{t}\prec\bk{f}{g}$ in $\omega'$. In order for the relation $\bk{s}{t}\prec\bk{f}{g}$ to exist in $\omega'$, the descending run corresponding to $\bk{f}{g}$ must be to the right of the descending run corresponding to $\bk{i}{j}$ in $\lambda(\omega)$. Since $\bk{u}{v}$ and $\bk{f}{g}$ are incomparable in $\omega$, this implies that $v<f$. Thus the claim is true. Therefore, all descending runs to the right of the descending run corresponding to $\bk{u}{v}$ in $\lambda(\omega)$ will remain to the right in $\lambda(\omega')$ and all descending runs to the left in $\lambda(\omega)$ will remain to the left in $\lambda(\omega')$. Thus $\bk{u}{v}$ will retain the placement $x$ in $\omega'$, i.e.\ $pl_{\omega}(\bk{u}{v})=pl_{\omega'}(\bk{u}{v})$. \\
\indent Now let $\bk{u}{v}$ be a block in $\omega$ with $pl_{\omega}(\bk{u}{v})>c$. By similar arguments, all descending runs to the right of the descending run corresponding to $\bk{u}{v}$ in $\lambda(\omega)$ will remain to the right in $\lambda(\omega')$ and all descending runs to the left in $\lambda(\omega)$ will remain to the left in $\lambda(\omega')$.  Since the descending runs corresponding to $\bk{i}{j}$ and $\bk{k}{l}$ have been combined to form a single descending run, there is one less descending run to the left of the descending run corresponding to $\bk{u}{v}$. Thus $pl_{\omega'}(\bk{u}{v})=pl_{\omega}(\bk{u}{v})-1$. We can now also conclude that if $a\leq pl_{\omega}(B^*)\leq c$, then $a\leq pl_{\omega'}(B^*)\leq c-1$.\\
\emph{Proof of (4)}: Suppose, for the purpose of contradiction, $B_1$ and $B'$ are not combinable in $\omega$. Then $B_1$ and $B'$ are comparable by a chain of covers in $\omega$, i.e.\ either $B_1\precdot A_1\precdot A_2\precdot\cdots\precdot A_k\precdot B'$ or $B'\precdot A_1\precdot A_2 \precdot\cdots\precdot A_k\precdot B_1$ for $k\geq 1$. Since $B_1$ and $B_2$ are combinable in $\omega$ they must either be incomparable or related by a cover in $\omega$. If $B_1$ and $B_2$ are incomparable in $\omega$ then we would have either $B\precdot A_1 \precdot A_2 \precdot \cdots \precdot A_k\precdot B'$ or $B'\precdot A_1\precdot A_2\precdot \cdots \precdot B$ in $\omega'$. Thus $B$ and $B'$ would not be combinable in $\omega'$, which is a contradiction. Therefore $B_1$ and $B_2$ are related by a cover in $\omega$. If $B_1$ and $B_2$ are related by a cover then $k=1$ and $B_2=A_1$ in the chain of covers since $B$ and $B'$ are combinable in $\omega'$. Thus we have either $B_1\precdot B_2\precdot B'$ or $B'\precdot B_2\precdot B_1$ in $\omega$. $B'\precdot B_2\precdot B_1$ is not possible because $pl_{\omega}(B_2)>pl_{\omega}(B_1)$. $B_1\precdot B_2\precdot B'$ implies by part \ref{labelorder} of this Proposition that $pl_{\omega'}(B')\geq c$, which is a contradiction.  Therefore $B_1$ and $B'$ are combinable in $\omega'$.\\ 
\emph{Proof of (5)}: There are two cases to consider. In the first case suppose that $B^*$ is $B_1$ (it cannot be $B_2$ due to its placement). Then $B'=B$. By part \ref{labelorder} of this Proposition, $a=b\leq pl_{\omega'}(B')\leq c-1$. In the second case suppose that $B^*$ is neither $B_1$ nor $B_2$. Then $B'=B^*$. If $b<a$, then by part \ref{labelorder} of this Proposition, $pl_{\omega'}(B')=b$. If $a<b<c$, then again by part \ref{labelorder} of this Proposition, $pl_{\omega'}(B')\leq c-1$.
\end{proof} 

\begin{lemma}
Let $B$ and $B'$ be blocks in $\omega\in(\Omega,\leq_S)$ with $pl_{\omega}(B)=a$ and $pl_{\omega}(B')=a+1$. Then $B$ and $B'$ are combinable in $\omega$. 
\end{lemma}

\begin{proof}
If $B$ and $B'$ are incomparable in $\omega$ then they are combinable and we are done. If $B$ and $B'$ are comparable in $\omega$ suppose that one block does not cover the other. Then there exists a block $C$ such that $B\prec C\prec B'$ in $\omega$. This implies that $pl_{\omega}(B)<pl_{\omega}(C)<pl_{\omega}(B')$, which is a contradiction. Thus $B$ and $B'$ are related by a cover and are therefore combinable. 
\end{proof}

\begin{lemma}\label{coverlemma}
Suppose $\omega$ and $\omega'$ are permutation pre-orders with $\omega<_S\omega'$. Let $B_1$ and $B_2$ be combinable blocks in $\omega$ such that they are contained in the same block in $\omega'$. Then there exists a permutation pre-order $\omega''$ with $\omega\lessdot_S\omega''\leq_S\omega'$ such that $B_1$ and $B_2$ are in the same block in $\omega''$.
\end{lemma}

\begin{proof}
Label the elements of $B_1$ as $a_1, a_2,\ldots, a_m$ with $a_i < a_{i+1}$ for $1\leq i\leq m-1$ and the elements of $B_2$ as $b_1, b_2, \ldots, b_n$ with $b_j< b_{j+1}$ for $1\leq j\leq n-1$. Without loss of generality suppose that $a_1<b_1$. Since blocks $B_1$ and $B_2$ are in the same block in $\omega'$ we know that $a_i\equiv b_j$ in $\omega'$ for $i\in[m], j\in[n]$. Define $\omega^*$ to be the permutation pre-order on $[n]$ with one two element block $\bk{a_1}{b_1}$ and the remaining elements in singleton blocks $\bk{v}{v}$ for $v\in[n]\setminus\{a_1,b_1\}$ with the following order relations:\\

(1) for $v\in[a_1,b_1]: \begin{cases} \text{if}\ a_1\preceq v\ \text{in}\ \omega', \text{then}\ a_1\preceq v\ \text{in}\ \omega^* \\
\text{if}\ a_1\succeq v\ \text{in}\ \omega', \text{then}\ a_1\succeq v\ \text{in}\ \omega^* \\
\text{if}\ a_1\equiv v\ \text{in}\ \omega', \text{then}\ a_1\preceq v\ \text{in}\ \omega^* \end{cases}$\\
\indent ((P1) rules out the possibility that $a_1$ and $v$ are incomparable.)\\

(2) for $v\notin [a_1,b_1]$: $v$ is incomparable in $\omega^*$ to $x$ for all $x\in [n]\setminus \{v\}$.\\

\indent Clearly, $\omega^*\in\Omega^{\Sigma}$. Taking the transitive closure of the union of $\omega$ with $\omega^*$ gives the permutation pre-order $\omega''$. 
\end{proof}

\begin{lemma}\label{min}
Given $\omega$ and an interval $(\omega,\omega')\in(\Omega,\leq_S)$, there exists a unique pair of blocks $B_1$ and $B_2$ whose larger placement is minimal among all pairs in $T(\omega,\omega')$. Furthermore, there exists a unique cover of $\omega$ in which $B_1$ and $B_2$ are combined.
\end{lemma}

\begin{proof}
Suppose for the purpose of contradiction there are two pairs of blocks with minimal larger placement in $\omega$. Thus we have placements $a, b$ and $c$ such that $a<b<c$ and the blocks with placements $a$ and $c$ and the blocks with placements $b$ and $c$ are combinable in $\omega$. For convenience we will denote the blocks by their respective placements. If blocks $a$ and $b$ are incomparable in $\omega$ or block $b$ covers block $a$ in $\omega$, then they are combinable in $\omega$. Thus the pair $(a,b)$ is in $T(\omega,\omega')$, contradicting the minimality of $c$. Block $a$ cannot be greater than block $b$ in $\omega$ due to their placements, so the only case left to consider is when block $b$ is greater than block $a$ but does not cover it. In this case there exists a chain of covers from $a$ to $b$ such that all of the blocks in the chain have placements less than $b$, i.e.\ $a\precdot a_1\precdot a_2 \precdot \cdots \precdot a_k \precdot b$ where $a_j< b$ for all $j\in[k]$. Therefore $(a,a_1)$ is in $T(\omega,\omega')$, contradicting the minimality of $c$.\\
\indent We have established that there exists a unique pair of blocks $B_1$ and $B_2$ whose larger placement $c$ is minimal among all pairs in $T(\omega,\omega')$. We know by Lemma \ref{coverlemma} that there exists at least one cover of $\omega$ in which $B_1$ and $B_2$ are combined. Suppose, again for the purpose of contradiction, there is more than one cover that can be obtained by combining $B_1$ and $B_2$ in $\omega$. The only way for this to occur is if a block is incomparable to both $B_1$ and $B_2$ in $\omega$ and comparable to the combined block in the covers. Thus we have the following situation. $B_1$ and $B_2$ are combined to form block $B$ in $\omega_1$ and $\omega_2$ for $\omega\lessdot_S\omega_1$ and $\omega\lessdot_S\omega_2$. Block $B'$ is incomparable to both $B_1$ and $B_2$ in $\omega$. In $\omega_1$, $B\prec B'$ and in $\omega_2$, $B\succ B'$ implying that $B'$ and $B$ are in the same block in $\omega'$.  This means that $B'$ is also in the same block as $B_1$ and $B_2$ in $\omega'$. By Proposition \ref{superprop} part \ref{btwn2} we know that $pl_{\omega}(B')=b$ such that $a<b<c$. Since $B'$ and $B_1$ are incomparable in $\omega$ they are a pair in $T(\omega,\omega')$, which contradicts the minimality of $c$. Therefore, there exists a unique cover of $\omega$ in which $B_1$ and $B_2$ are combined.
\end{proof}

We now have the necessary tools to prove Theorem \ref{ELlabel}.

\begin{proof}
We need to establish that $\sigma$ satisfies conditions (EL1) and (EL2). Let $(\omega,\omega')$ be an interval in $(\Omega,\leq_S)$. Denote by $\zeta$ the maximal chain $\omega=\omega_0\lessdot_S\omega_1\lessdot_S\cdots\lessdot_S\omega_k=\omega'$ where each $\omega_{i+1}$ is obtained by combining a pair of blocks from $T(\omega_i,\omega')$ such that the larger of the two placements is minimal among all pairs. By Lemma \ref{min}, this maximal chain is unique and thus does not share its label sequence with any other chain.\\
\indent Now we will show that the edge label sequence for $\zeta$ is weakly increasing. Suppose, for the purpose of contradiction, there is a descent in the edge label sequence, i.e.\ $\sigma(\omega_j,\omega_{j+1})=c$ and $\sigma(\omega_{j+1},\omega_{j+2})=b$ for $c>b$. Let $B$ and $B'$ with $pl_{\omega_{j+1}}(B)=a<b=pl_{\omega_{j+1}}(B')$ be the blocks in $\omega_{j+1}$ that were combined to form $\omega_{j+2}$. We claim that the pair of blocks combined in $\omega_j$ did not have minimal larger placement.\\
\emph{Proof of claim:} There are three cases to consider. In the first case we suppose that $B$ and $B'$ are the same blocks in $\omega_j$ and $\omega_{j+1}$. Since blocks $B$ and $B'$ are not the blocks that were combined in $\omega_j$ to form $\omega_{j+1}$ and they both have placements less than $c$ in $\omega_{j+1}$, Proposition \ref{superprop} part \ref{labelorder} tells us that $pl_{\omega_j}(B)<c$ and $pl_{\omega_j}(B')<c$. The only way in which $B$ and $B'$ are not combinable in $\omega_j$ but are combinable in $\omega_{j+1}$ is if $B\precdot B^*\precdot B'$ in $\omega_j$ and either $B$ and $B^*$ or $B^*$ and $B'$ are combined to form $\omega_{j+1}$. Since we are assuming that $B$ and $B'$ are the same blocks in $\omega_j$ and $\omega_{j+1}$, this cannot occur. This implies that $B$ and $B'$ are combinable in $\omega_j$. Thus $B$ and $B'$ are a combinable pair of blocks in $\omega_j$ with larger placement less than $c$. In the second case, suppose that $B$ is the new block in $\omega_{j+1}$ formed by combining blocks $B_1$ and $B_2$ in $\omega_j$ where $pl_{\omega_j}(B_1)=c_1<c=pl_{\omega_j}(B_2)$. Then by Proposition \ref{superprop} part \ref{comb}, $B'$ and $B_1$ are a combinable pair in $\omega_j$ with larger placement less than $c$. In the third case, suppose that $B'$ is the new block in $\omega_{j+1}$ formed by combining blocks $B'_1$ and $B'_2$ where $pl_{\omega_j}(B'_1)=c_2<c=pl_{\omega_j}(B'_2)$. Then again by Proposition \ref{superprop} part \ref{comb}, $B$ and $B'_1$ are a combinable pair in $\omega_j$ with larger placement less than $c$. In each case, we obtain a pair of blocks in $\omega_j$ such that the larger placement of the pair is less than $c$. Therefore, the claim is true and we have reached a contradiction. Thus, there cannot be a descent in the edge label sequence for $\zeta$, meaning $\zeta$ is weakly increasing.\\
\indent All that remains to be shown for (EL1) is that $\zeta$ is the only weakly increasing maximal chain. As established above, $\zeta$ does not share its label sequence with any other chain. Suppose we follow some other maximal chain $\chi$ from $\omega$ to $\omega'$. At some step these chains must differ, meaning that in $\zeta$ we combine two blocks from $\omega_j$ to form $\omega_{j+1}$ such that the larger placement of the pair of blocks in $\omega_j$ is minimal among pairs in $T(\omega_j,\omega')$, whereas in $\chi$ we combine two blocks from $\omega_j$ to form $\omega_{j+1}^*$ such that the larger placement of the pair of blocks in $\omega_j$ is not minimal among pairs in $T(\omega_j,\omega')$. Thus we have the following situation. Let $B_1$ and $B_2$ be blocks of $\omega_j$ with $pl_{\omega_j}(B_1)=a<c=pl_{\omega_j}(B_2)$ that were combined to form $\omega_{j+1}$. Let $B^*_1$ and $B^*_2$ be blocks of $\omega_j$ with $pl_{\omega_j}(B^*_1)=b<d=pl_{\omega_j}(B^*_2)$, for $d>c$, that were combined to form $\omega_{j+1}^*$. Since $B_1$ and $B_2$ are a pair in $T(\omega_j,\omega')$, they must be in the same block in $\omega'$. Thus, at some later step in $\chi$, a block containing $B_1$ must be combined with a block containing $B_2$. By Proposition \ref{superprop} part \ref{staysmall}, if we continue to pick pairs of blocks that cause the label sequence of $\chi$ to increase, the placements of the blocks containing $B_1$ and $B_2$ in $\omega_k^*$ for $k>j$ will always be smaller than the edge label leading to $\omega_k^*$ assigned by $\sigma$ (because in $\omega_{j+1}^*$ $B_1$ and $B_2$ have edge labels less than $d$). Eventually we are forced to combine the blocks containing $B_1$ and $B_2$ (or a pair with smaller larger placement) creating a descent in the edge label sequence for $\chi$. Thus, all other maximal chains have at least one descent. Therefore, $\zeta$ satisfies (EL1).\\
\indent Since we chose the smallest possible label at each step in $\zeta$, its label sequence is lexicographically smaller than the label sequences of all other maximal chains. Therefore, $\zeta$ satisfies (EL2).
\end{proof}

Since $(\Omega,\leq_S)$ is a graded poset, Theorem \ref{ELlabel} implies that it is EL-shellable and hence shellable \cite[Theorem 2.3]{aB80}. Thus the M\"{o}bius number of an interval in $(\Omega,\leq_S)$, $\mu(\omega,\omega')$, is equal to $(-1)^{|\omega|-|\omega'|}$ times the number of strictly decreasing maximal chains from $\omega$ to $\omega'$.  To count the strictly decreasing maximal chains from $\omega$ to $\omega'$, we need to determine the number of ways in which the combinable block with the highest placement can be combined with blocks with lower placements at each step in the chain. At each step in a maximal chain there may be multiple ways to go up by a cover and obtain the maximum possible label on that edge. For this reason, counting strictly decreasing maximal chains is not completely straightforward, and at this time we do not know how to count strictly decreasing maximal chains for general intervals. Interestingly, the M\"{o}bius number of the entire lattice has a simple description \cite{nR09}. It is the number of indecomposable permutations, or equivalently, the number of permutations with no global descents. See Sequence A003319 of \cite{njaS} for details and references.

\section{Noncrossing Pre-orders}\label{sectnon}

In this final section, we consider a subset of $S_n$, the $c$-sortable permutations, and describe the corresponding subset of the permutation pre-orders.\\
\indent A Coxeter group $W$ is generated by a set $S$ of \emph{simple generators}. For $W=S_n$ the simple generators are $S=\{s_1,s_2,\ldots,s_{n-1}\}$ where $s_i=(i,i+1)$ for $1\leq i \leq n-1$. A \emph{Coxeter element} $c$ is the product of the simple generators in any order. For a chosen Coxeter element $c\in S_n$, we can define a \emph{barring} on the numbers $\{2,3,\ldots,n-1\}$ as follows: 
\begin{enumerate}
\item if $s_{i-1}$ is before $s_{i}$ in $c$ then $i$ is lower-barred, and we denote this by $\underline{i}$
\item if $s_{i-1}$ is after $s_{i}$ in $c$ then $i$ is upper-barred, and we denote this by $\overline{i}$.
\end{enumerate}

A Coxeter element in $S_n$ can be written as an $n$-cycle using the barring by placing $1$ at the 12 o'clock position on a circle followed clockwise by the lower-barred numbers in ascending numerical order, the number $n$, and then the upper-barred numbers in descending numerical order.

\begin{figure}[ht]
\[\raisebox{-27 pt}{\includegraphics{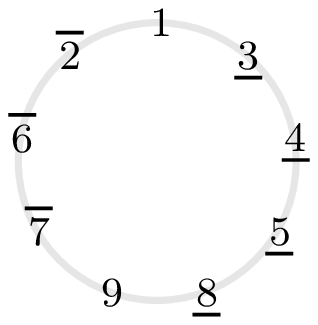}}
\]
\caption{}\label{fig:cycle}
\end{figure}

\begin{example}\label{barringex}
For $c=s_2s_1s_3s_7s_6s_4s_5s_8\in S_9$ the lower-barred numbers are $\underline{3},\underline{4},\underline{5}, \underline{8}$ and the upper-barred numbers are $\overline{2},\overline{6},\overline{7}$. This gives us the cycle shown in Figure \ref{fig:cycle}.
\end{example}

For a permutation $\pi\in S_n$ and $x\in S_m$ we will say that $\pi$ \emph{contains the pattern $x$} if there are integers $1\leq i_1< i_2<\cdots<i_m\leq n$ such that for all $1\leq j <k \leq m$ we have $x_j<x_k$ if and only if $y_{i_j}<y_{i_k}$. Otherwise we will say that $\pi$ \emph{avoids the pattern $x$}. In this paper we will be concerned with patterns involving barred numbers, specifically the patterns $\overline{2}31$ and $31\underline{2}$. For $\pi\in S_n$ with chosen Coxeter element $c$, we say that $\pi$ \emph{contains the pattern} $\overline{2}31$ if it contains an instance of the pattern $231$ in which the number representing the $2$ in $\pi$ is upper-barred. Similarly, we say $\pi$ \emph{contains the pattern} $31\underline{2}$ if it contains an instance of the pattern $312$ in which the number representing the $2$ in $\pi$ is lower-barred. If these conditions do not hold, we again say that $\pi$ \emph{avoids the given pattern}.

\begin{example}
As a continuation of Example \ref{barringex}, suppose $c=s_2s_1s_3s_7s_6s_4s_5s_8\in S_9$. Then $\pi=163425897$ contains four instances of the pattern $31\underline{2}$: $63\underline{4}$, $63\underline{5}$, $64\underline{5}$, and $62\underline{5}$, and avoids the pattern $\overline{2}31$. 
\end{example} 

There is a general definition of $c$-sortable elements of a Coxeter group which can be found in \cite{nR07}. For this paper, the following characterization of $c$-sortable elements in $S_n$ is sufficient. Given a Coxeter element $c$, a permutation $\pi\in S_n$ is \emph{$c$-sortable} if and only if $\pi$ avoids the patterns $\overline{2}31$ and $31\underline{2}$ \cite[Lemmas 4.1 and 4.8]{nR07}.\\
\indent A \emph{$c$-noncrossing partition} is a partition of $[n]$ into blocks such that when the blocks are drawn on the cycle $c$, the convex hulls of the blocks do not overlap.

\begin{figure}[ht]
\[\raisebox{-27 pt}{\includegraphics{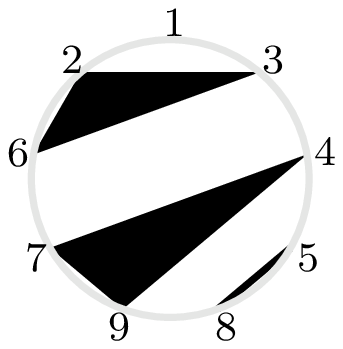}}
\]
\caption{}\label{fig:noncross}
\end{figure}

\begin{example}
Continuing Example \ref{barringex}, a $c$-noncrossing partition of the cycle $c$ is shown in Figure \ref{fig:noncross}.
\end{example} 

Define a \emph{noncrossing pre-order} to be a pre-order $\omega\in\Omega$ on $[n]$ with respect to a Coxeter element $c$ with the following two conditions:
\begin{enumerate}
\item blocks in $\omega$ form a $c$-noncrossing partition of the cycle $c$, and 
\item suppose blocks $B_{[i_s,i_1]}$ and $B_{[j_t,j_1]}$ intersect in $\omega$, so that, without loss of generality, there exists a $j_x\in B_{[j_t,j_1]}$ for $x\in[t]$ such that $i_s<j_x<i_1$. If $j_x$ is upper-barred with respect to $c$ then $B_{[i_s,i_1]}\prec B_{[j_t,j_1]}$ in $\omega$ and if $j_x$ is lower-barred with respect to $c$ then $B_{[i_s,i_1]} \succ B_{[j_t,j_1]}$ in $\omega$. 
\end{enumerate} Let $\Omega_c^{NC}=\{\omega\in\Omega: \omega\ \textrm{is a noncrossing pre-order}\}$, and let $\Omega^{c}$ be the image of the $c$-sortable permutations under $\mu$.

\begin{proposition}
$\Omega^c=\Omega_c^{NC}$
\end{proposition}

\begin{proof}
For any $c$-noncrossing partition there is a unique partial order on the blocks which defines a noncrossing pre-order. This defines a bijection from $\Omega_c^{NC}$ to $c$-noncrossing partitions. In \cite[Theorem 6.1]{nR07} it was shown that there exists a bijection between $c$-noncrossing partitions and $c$-sortable permutations. Thus $|\Omega^c|=|\Omega_c^{NC}|$, and it is enough to show that $\Omega^c\subseteq\Omega_c^{NC}$. Let $\omega\in\Omega^c$ and let $\lambda(\omega)=\pi$. Suppose $i_1\cdots i_s$ and $j_1\cdots j_t$ are distinct descending runs in $\pi$ corresponding respectively to blocks $B_{[i_s,i_1]}$ and $B_{[j_t,j_1]}$ in $\omega$ such that $B_{[i_s,i_1]}$ and $B_{[j_t,j_1]}$ overlap. Then without loss of generality there exists a $j_x\in B_{[j_t,j_1]}$ such that $i_s<j_x<i_1$ for $x\in[s]$. Suppose $j_x$ is upper-barred with respect to $c$. Then if $B_{[i_s,i_1]}\succ B_{[j_t,j_1]}$ in $\omega$, $\pi=\cdots j_1\cdots j_t\cdots i_1\cdots i_s\cdots$ and $j_xi_1i_s$ forms a $\overline{2}31$ pattern, which is a contradiction of $\pi$ being $c$-sortable. Thus $B_{[i_s,i_1]}\prec B_{[j_t,j_1]}$ in $\omega$. Suppose $j_x$ is lower-barred with respect to $c$. Then if $B_{[i_s,i_1]}\prec B_{[j_t,j_1]}$, $\pi=\cdots i_1\cdots i_s\cdots j_1\cdots j_t\cdots$ and $i_1i_sj_x$ forms a $31\underline{2}$ pattern which is a contradiction of $\pi$ being $c$-sortable. Thus $B_{[i_s,i_1]}\succ B_{[j_t,j_1]}$ in $\omega$. 
\end{proof} 

The choice $c=s_1s_2\cdots s_{n-1}$ is the case where every number is lower-barred and the choice $c=s_{n-1}s_{n-2}\cdots s_1$ is the case where every number is upper-barred. In these specific cases, if two distinct blocks $B_{[i,j]}$ and $B_{[k,l]}$ overlap in $\omega$, then either $[i,j]\subset[k,l]$ or $[k,l]\subset[i,j]$. Thus we have the following corollaries.

\begin{corollary} Let $c=s_1s_2\cdots s_{n-1}$ and let $\omega\in\Omega_c^{NC}$. Then $B_{[i,j]}\preceq B_{[k,l]}$ in $\omega$ if and only if $[i,j]\subset[k,l]$.
\end{corollary}

\begin{corollary}
Let $c=s_{n-1}s_{n-2}\cdots s_1$ and let $\omega\in\Omega_c^{NC}$. Then $B_{[i,j]}\preceq B_{[k,l]}$ in $\omega$ if and only if $[k,l]\subset[i,j]$.
\end{corollary}

\section*{Acknowledgments} The author thanks her advisor Dr. Nathan Reading for presenting her with the idea to explore the shard intersection order on permutations and for insight and direction along the way.

\end{document}